\documentclass[10pt,twoside]{amsart}
\usepackage{amssymb}
\usepackage{amsmath}
\usepackage{amsthm}
\usepackage{mathrsfs}

\usepackage{tikz}

\usepackage{enumitem}

\usepackage[T1]{fontenc}

\newcommand{\n}{\cap}
\newcommand{\un}{\cup}
\newcommand{\R}{\mathbb R}
\newcommand{\C}{\mathbb C}

\newcommand{\Z}{\mathbb Z}
\newcommand{\N}{\mathbb N}
\newcommand{\F}{\mathbb F}

\newcommand{\CC}{\mathcal C}

\renewcommand{\epsilon}{\varepsilon}
\renewcommand{\phi}{\varphi}

\newcommand{\recip}[1]{\frac{1}{#1}}

\renewcommand{\Re}{\operatorname{Re}}
\renewcommand{\Im}{\operatorname{Im}}
\DeclareMathOperator{\im}{im}

\DeclareMathOperator{\dist}{dist}

\renewcommand{\:}{\colon}
\newcommand\bottomnote[1]{%
  \begingroup
  \renewcommand\thefootnote{}\footnote{#1}%
  \addtocounter{footnote}{-1}%
  \endgroup
}
\newcommand{\satref}[1]{\emph{\ref{#1}}}
\newcommand{\MSC}[1]{\href{http://www.ams.org/msc/msc2010.html?t=#1&btn=Current}{#1}}

\theoremstyle{plain}
\newtheorem{thm}{Theorem}
\newtheorem{lem}[thm]{Lemma}
\newtheorem{prop}[thm]{Proposition}

\theoremstyle{definition}
\newtheorem{defn}[thm]{Definition}
\theoremstyle{remark}
\newtheorem{rk}[thm]{Remark}

\theoremstyle{plain}

\newcounter{examplesubsection}[thm]
\renewcommand{\theexamplesubsection}{\thethm.\arabic{examplesubsection}}

\newcounter{extsection}
\renewcommand{\theextsection}{\arabic{extsection}}

\usepackage[breaklinks=true]{hyperref}
\usepackage[initials,short-journals]{amsrefs}

\hypersetup{
      colorlinks=true, 
      linkcolor=[RGB]{0, 0, 0}, 
      citecolor=[RGB]{0, 0, 0}, 
      filecolor=magenta, 
      urlcolor=[RGB]{0,0,128}, 
}

\numberwithin{equation}{section}
\numberwithin{thm}{section}

\allowdisplaybreaks
\title{PATHOLOGICAL PHENOMENA IN DENJOY-CARLEMAN CLASSES}
\author{Ethan Y. Jaffe}

\begin{document}

\begin{abstract}
Let $\CC^M$ denote a Denjoy-Carleman class of $\CC^\infty$ functions (for a given logarithmically-convex sequence $M = (M_n)$). We construct: (1) a function in $\CC^M((-1,1))$ which is \emph{nowhere} in any smaller class; (2) a function on $\R$ which is formally $\CC^M$ \emph{at every point}, but not in $\CC^M(\R)$;  (3) (under the assumption of quasianalyticity) a smooth function on $\R^p$ ($p \geq 2$) which is $\CC^M$ on every $\CC^M$ curve, but not in $\CC^M(\R^p)$.
\end{abstract}

\maketitle
\tableofcontents

\section{Introduction}
\bottomnote{2010 \textit{Mathematics Subject Classification.} Primary: \MSC{26E10}, \MSC{26B35}; Secondary: \MSC{26E05}, \MSC{30D60}, \MSC{46E25}}
\bottomnote{\textit{Key words and phrases.} Denjoy-Carleman classes, quasianalytic functions, quasianalytic curve, arc-quasianalytic}
\bottomnote{Research supported in part by an NSERC Undegraduate Student Research Award and NSERC grant OGP0009070.}
The aim of this article is to provide explicit constructions of several examples of functions illustrating pathologies and subtleties in the theory of Denjoy-Carleman classes. In the following, $\mathbb F$ will denote either $\R$ or $\C$. The first example is of a function in any given Denjoy-Carleman class, but not in any smaller Denjoy-Carleman class. 
\begin{thm}\label{thm:nosmallerclass}For any Denjoy-Carleman class $\CC^M$ there exists $f \in \CC^\infty((-1,1),\F)$ satisfying\emph{:}

\begin{enumerate}[label=\emph{(\arabic*)}]
\item \label{sat:nosmallerclass:CM} $f \in \CC^M((-1,1),\F)$\emph{;}
\item \label{sat:nosmallerclass:notCN} for any Denjoy-Carleman class $\CC^N \subsetneq \CC^M$, and any open subset $U \subseteq (-1,1)$, $f \not \in \CC^N(U)$.
\end{enumerate}
\end{thm}

The second example is of a function which is formally in a given Denjoy-Carleman class at all points, but is nonetheless not in that class (the notation $f \in \mathcal F^M(x,\mathbb F)$ indicates that $f$ is formally of class $\CC^M$ at $x$; see Definition \ref{def:formally}):
\begin{thm}\label{thm:formaleverywherenot}Let $\CC^M$ be any Denjoy-Carleman class. Then, there exists $f \in \CC^\infty(\R,\F)$ satisfying\emph{:}
\begin{enumerate}[label = \emph{(\arabic*)}]
\item \label{sat:formaleverywherenot:CM}$f \in \CC^M(\R\setminus\{0\},\F)$\emph{;}
\item \label{sat:formaleverywherenot:formal}$f \in \mathcal F^M(0,\F)$\emph{;}
\item \label{sat:formaleverywherenot:notCM}$f \not \in \CC^M(\R,\F)$.
\end{enumerate}
\end{thm}
We remark that if $f \in \CC^\infty(U,\F)$, where $U \subseteq \R^p$ is open, and $f \in \mathcal F(x,\F)$ for all $x \in U$, then there is an open dense subset $V$ of $U$ such that $f \in \CC^M(V,\F)$ (see Proposition~\ref{prop:formalmeansalmostCM}).

Like the second example, the third example is ``close'' to being $\CC^M$, but is not actually: it is smooth and its composition with every quasianalytic curve of a given quasianalytic Denjoy-Carleman class is in the class, yet is not itself in the class.
\begin{thm}\label{thm:bigtheorem}For any $p \geq 2$ and any quasianalytic Denjoy-Carleman class $\CC^M$, which is not the class of analytic functions, there exists $f \in \CC^\infty(\R^p)$ such that for any curve $\gamma \in \CC^M(U,\R^p)$ (where $U \subseteq \R$ is open), $f\circ\gamma \in \CC^M(U,\F)$, but $f \not \in \CC^M(\R^p,\F)$.
\end{thm}
\par Theorem~\ref{thm:bigtheorem} follows easily from the following result:
\begin{thm}\label{thm:masterthm}For any $p \geq 2$ and any Denjoy Carleman class $\CC^M$, which is not the class of analytic functions, there exists $f \in \CC^\infty(\R^p,\F)$ satisfying\emph{:}
\begin{enumerate}[label = \emph{(\arabic*)}]
\item \label{sat:masterthm:CM}$f \in \CC^M(\R^p\setminus\{0\},\F)$\emph{;}
\item \label{sat:masterthm:polyCM}for any $a > 0$ and integer $m \geq 1$, $f \in \CC^M(\mathcal S_{a,m}^p,\F)$\emph{;}
\item \label{sat:masterthm:quadCM}$f \in \CC^M(\R^p \setminus \mathcal Q^p,\F)$\emph{;}
\item \label{sat:masterthm:notCM}$f \not \in \CC^M(\R^p,\F)$,
\end{enumerate}
where \[\mathcal S_{a,m}^p := \{x=(x_1,x_2,\ldots,x_p) \in \R^p \: x_1 \geq 0 \text{ and } x_2 \geq ax_1^m\}\]
and 
\[\mathcal Q^p := \{x=(x_1,x_2,\ldots,x_p) \in \R^p \: x_1 > 0 \text{ and } x_2 > 0\}.\]
\end{thm}

Denjoy-Carleman classes have been classically studied in their relation to PDE theory, harmonic analysis, and other fields. Recently, there has been renewed interest in these classes from a more analytic-geometric viewpoint. The theory of Denjoy-Carleman classes is usually divided into the study of \emph{quasianalytic} classes, characterized by an analogue of analytic continuation: all the derivatives at a point of a function in such a class uniquely determines the function (at least locally), and \emph{non-quasianalytic} classes.

However, despite quasianalytic classes satisfying ``quasianalytic continuation'', their theory remains not well-understood. This is in a large because many standard techniques for analytic functions, namely the Weierstrass division and preparation theorems, fail in general for quasianalytic Denjoy-Carleman classes (see \cites{ACQ, ChauCho, Childress, ParRol, Thilliez}). This makes deciding whether these classes are Noetherian very difficult.

In relation to Theorem~\ref{thm:nosmallerclass}, several results are known. It is a classical result that each Denjoy-Carleman class contains functions which are not in any smaller class \cite[Thm.\,1]{Thilliez}. More recently, \cite[Thm.\,2]{RSW} shows that there is a function in a given \emph{quasianalytic} Denjoy-Carleman ring which is nowhere analytic. This was proven by examining ``lacunarity'' properties of Fourier series. Theorem~\ref{thm:nosmallerclass} can be seen as a strenghthening of the conclusion of the first result and as a generalization of the second.

By a classical theorem of Carleman (see \cite[Thm.\,3]{Thilliez}), there is a smooth function germ which is formally quasianalytic of a given class, but does not correspond to any actual quasianalytic function germ of the same class. Recently, another example of such a non-extendable function was constructed in \cite[Thm.\,1.2]{ACQ}. Like these examples, the function of Theorem~\ref{thm:formaleverywherenot} is formally of a given Denjoy-Carleman class, yet fails to be of actually of the class. There are two main differences between Theorem~\ref{thm:formaleverywherenot} and both Carleman's function and that of \cite[Thm.\,1.2]{ACQ}: Theorem~\ref{thm:formaleverywherenot} involves arbitrary Denjoy-Carleman classes instead of quasianalytic classes, but does not consider the question of whether the germ is extendable. In fact, in the so-called \emph{strongly non-quasianalytic} case, the function must be extendable (\cite[Thm.\,4]{Thilliez}). Furthermore, the function constructed in \cite{ACQ} is formally in the given Denjoy-Carleman class only on $[0,\infty)$, whereas that of Theorem~\ref{thm:formaleverywherenot} is formally in the given Denjoy-Carleman class on the entire real line.

Given certain classes $\CC$ of real- or complex-valued functions of several real variables, it is a natural to consider whether a function $f$, is of class $\CC$ provided that $f$ is of class $\CC$ on every curve of class $\CC$. In \cite{Boman}, Boman considers the question in the case $\CC = \CC^\infty$, and answers it in the affirmative. In \cite{BMAnal}, Bierstone, Milman, and Parusi\'{n}ski answered the question in the negative for the class of analytic functions, showing that a function which is analytic on every analytic curve (a so-called ``arc-analytic function'') is not necessarily even continuous. In fact, their example works for any class of quasianalytic functions. In \cite[Thm.\,3.9]{RainerNQuas} and \cite[Thm.\,2.7]{RainerQuas}, Kriegly, Michor, and Rainer answer the problem in the affirmative where $\CC = \CC^M$ is a \emph{non}-quasianaltyic Denjoy-Carleman class. In \cite{RainerQuas} they also raise the question, if $\CC^M$ is a quasianalytic Denjoy-Carlemean, whether a \emph{smooth} function which is of class $\CC^M$ along each $\CC^M$ curve is of class $\CC^M$. Theorem~\ref{thm:bigtheorem} answers this questions, and provides an example of a function which is smooth, and quasianalytic of a given class $\CC^M$ on every $\CC^M$ curve (called ``arc-quasianalytic'' in \cite{BMarc}), yet not itself $\CC^M$.

The author's research was conducted as an NSERC Undergraduate Summer Research project under the supervision of Edward Bierstone. The author would like to thank Dr. Bierstone for raising the question treated in Theorem~\ref{thm:bigtheorem} and for his numerous suggestions for this article. The author is grateful to both Dr. Bierstone and Andr\'{e} Belotto for helping him develop his ideas, and to Armin Rainer and David Nenning for pointing out helping correct numerous errors in the first draft of this article.

\section{Preliminaries}

Below we give several basic definitions.

$\N$ denotes the set of non-negative integers.
For a multi-index $\alpha = (\alpha_1,\ldots,\alpha_p) \in \N^p$, set:
\begin{align*}
|\alpha| &:= \alpha_1 + \cdots + \alpha_p\\
D^\alpha &:= \frac{\partial^{|\alpha|}}{\partial x_1^{\alpha_1} \cdots \partial x_p^{\alpha_p}}\\
\alpha! &:= \alpha_1!\cdots\alpha_p!
\end{align*}

If $p \geq 2$, we denote the Euclidean norm on $\R^p$ by \[||x|| = ||(x_1,\ldots,x_p)|| = \sqrt{x_1^2+\cdots+x_p^2}\ .\]
For any bounded subset $S \subseteq \R^p$, we write \[||S|| := \sup_{x \in S} ||x|| < \infty.\]
If $p \geq 1$, $t \in \R$, $a \in \R^p$, $S \subseteq \R^p$, we write \[tS\pm a := \{ts\pm a \: s \in S\}.\]

We also denote by $\mathcal C^\infty(U,\F)$ the $\F$-algebra of smooth (infinitely-differentiable) $\F$-valued functions on an open set $U \subseteq \R^p$, and $\CC^\infty$ the class of all smooth functions. Unless otherwise specified, we write $\CC^\infty(U)$ for $\CC^\infty(U,\C)$. Likewise, we denote by $\CC^\omega(U,\F)$ the corresponding algebra of analytic functions on $U$, and $\CC^\omega$ the class of all analytic functions. Unless otherwise specified, we write $\CC^\omega(U)$ for $\CC^\omega(U,\C)$.

Let $M = (M_n)_{n=0}^\infty$ be a non-decreasing sequence of positive real numbers with $M_0 = 1$.
\begin{defn} For an open set $U \subseteq \R^p$, we say that a function $f \in \CC^\infty(U,\F)$ belongs to the set $\CC^M(U,\F)$ if either of the following two equivalent conditions holds:
\begin{enumerate}[label = (\roman*)]
\item for any $x \in U$, there exists some open $V \subseteq U$ containing $x$ and constants $A,B > 0$ such that, for any multi-index $\alpha \in \N^p$ and $y \in V$
\begin{equation}\label{eq:DCdef} |D^\alpha f(y)| \leq AB^{|\alpha|}|\alpha|!M_{|\alpha|};\end{equation}
\item for any compact set $K \subseteq \R^p$ contained in $U$, there are $A,B > 0$, such that for all $y \in K$, \eqref{eq:DCdef} holds.
\end{enumerate}
In this case, we will say that $f$ is of class $\CC^M$. $\CC^M$ is called a ``Denjoy-Carleman'' class.
\end{defn}
\begin{rk}Note that if $M = (M_n)_{n=0}^\infty$ is identically $1$, then $\CC^M$ is the class $\CC^\omega$ of analytic functions. We will call a Denjoy-Carleman class $\CC^M$ ``non-analytic'' if $\CC^M \neq \CC^\omega$.
\end{rk}
\begin{defn}\label{def:formally} We say that a function $f \in \CC^\infty(U,\F)$ is \emph{formally} $\CC^M$ at a point $y \in U$, if there are $A,B > 0$ such that \eqref{eq:DCdef} holds; in this case we write $f \in \mathcal F^M(y,\F)$ (i.e. the coefficients of the formal power series of $f$ at $y$ satisfy bounds similar to those in \eqref{eq:DCdef}).
\end{defn}
\begin{defn} Given a closed subset $C \subseteq \R^p$, we say that $f:C\to\F$ is in $\CC^M(C,\F)$ if there is some open set $U \supseteq C$ such that $f \in \CC^\infty(U,\F)$, and, for each $x \in C$, there is an open neighbourhood $V$ containing $x$, such that \ref{eq:DCdef} holds for all $y \in V\n C$, with suitable $A,B>0$.
\end{defn}

For any open or closed $S \subseteq \R^p$, we write $\CC^M(S)$ for $\CC^M(S,\C)$. Likewise, we always write $\mathcal F^M(x)$ for $\mathcal F^M(x,\C)$.
\begin{rk}\label{rk:finitelmanydonotmatter}Note that in all of the above definitions, the requirement of having upper bounds on \emph{all} derivatives is actually equivalent to the apparently weaker requirement that there is an upper bound of the same form on \emph{all but finitely many} of the derivatives.\end{rk}

In order that Denjoy-Carleman classes satisfy useful properties, one imposes the condition that $M$ is \emph{logarithmically convex}, i.e. the ratios $M_{n+1}/M_n$ form a non-decreasing sequence. This condition implies  that the sequence $M_n^{1/n}$ is also non-decreasing (see \cite[\S1.3]{Thilliez}). Because of the Leibniz rule, logarithmic convexity implies that the sets $\CC^M(U,\F)$ are closed under multiplication (for $U$ open in $\R^p$). Since $\CC^M(U,\F)$ is also closed under addition, the logarithmic convexity of $M$ implies that $\CC^M(U,\F)$ forms an $\F$--subalgebra of $\CC^\infty(U,\F)$. For the remainder of this article, we work exclusively work with Denjoy-Carleman classes $\CC^M$, for $M$ logarithmically-convex.

It is also sometimes required that
\begin{equation}\sup_{n \geq 1} \left(\frac{M_{n+1}}{M_n}\right)^{1/n} < \infty.\end{equation}
This condition is equivalent to stablility under differentiation of $\CC^M(U,\F)$ (\cite[Cor.\,2]{Thilliez}). However, none of the results in this article assume this fact.

For two Denjoy-Carleman classes $\CC^M$ and $\CC^N$, $\CC^M(U,\F) \subseteq \CC^{N}(U,\F)$ if and only if
\begin{equation}\label{eq:insidecondition}\sup_{n \geq 1} \left(\frac{M_n}{N_n}\right)^{1/n} < \infty\end{equation}
(see \cite[\S1.4]{Thilliez}). In particular, $\CC^M(U,\F) = \CC^\omega(U,\F)$ if and only if $\sup_{n \geq 1} M_n^{1/n} < \infty$. We write $\CC^M \subseteq \CC^N$ if \eqref{eq:insidecondition} holds.
\begin{defn}A mapping $g:U\to \R^p$, where $U \subseteq \R^p$ is open, is said to be of class $\CC^M$ if each component function $g_i \in \CC^M(U,\R)$, where $g = (g_1,\ldots,g_p)$. In this case, we write $g \in \CC^M(U,\R^p)$.\end{defn}

\begin{thm}[see \cite{BM}*{Thm.\,4.7}]\label{thm:composition} Let $U \subseteq \R$ be open and suppose that $\gamma \in \CC^M(U,\R^p)$ and $f \in \CC^M(S,\F)$, where $S$ is an open or closed subset of $\R^p$ containing $\im(\gamma)$. Then the composite function $f \circ \gamma \in \CC^M(U,\F)$.\end{thm}

\begin{defn}A class $\CC$ of smooth functions is called \emph{quasianalytic} if whenever $U\subseteq\R^p$ is open and $f \in \CC(U,\F)$ satisfies $D^\alpha f(x) = 0$ for all $\alpha \in \N^p$ and some $x \in U$, then $f$ is identically 0 in a neigbourhood of $x_0$.\end{defn}
The Denjoy-Carleman theorem (\cite[Thm.\,1.3.8]{Hor}; also \cite[Thm.\,2]{Thilliez}) characterizes Denjoy-Carleman classes which are quasianalytic.
\begin{thm}[Denjoy-Carleman]
A Denjoy-Carleman class $\CC^M$ is quasianalytic if and only if
\[\sum_{n = 0}^\infty \frac{M_n}{(n+1)M_{n+1}} = \infty.\]
\end{thm}

\section{A function in a given Denjoy-Carleman class which is nowhere in any smaller class}
The example we construct here is based on the idea Borel used in \cite{Borel} to construct a class of quasianalytic functions that contains nowhere analytic functions. The example constructed here was inspired by, and uses several ideas in the construction of the non-extendable function of \cite[Thm.\,1.2]{ACQ}. The idea will be to construct the function as the restriction to $(-1,1)$ of a series of rational functions
\[\sum_{n=1}^\infty \frac{A_n}{z-z_n}.\]
where $z_n$ is a sequence of non-real complex numbers accumulating everywhere $(-1,1)$. Theorem~\ref{thm:nosmallerclass} will be proved using the following proposition:
\begin{prop}\label{prop:nosmallerclassbuild} For any non-analytic Denjoy-Carleman class $\CC^M$, there exists $f \in \CC^\infty((-1,1))$ satisfying\emph{:}
\begin{enumerate}[label = \emph{(\arabic*)}]
\item \label{sat:nosmallerclassbuild:upper}for all $j \geq 0$ and $x \in (-1,1)$, $|f^{(j)}(x)| \leq \frac{9}{2}j!M_j$\emph{;}
\item \label{sat:nosmallerclassbuild:lower}for any dyadic rational $x \in (-1,1)$, and large enough $j$, \[|f^{(j)}(x)| \geq \recip{2}\recip{3^j}j!M_j.\]
\item \label{sat:nosmallerclassbuild:realim}for any dyadic rational $x \in (-1,1)$ and large enough $j$, either \[|\Re(f)^{(j)}(x)| \leq \recip{3}|\Im(f)^{(j)}(x)| \text { or } |\Im(f)^{(j)}(x)| \leq \recip{3}|\Re(f)^{(j)}(x)|.\]
\end{enumerate}
\end{prop}
First, we will prove Theorem~\ref{thm:nosmallerclass} using Proposition~\ref{prop:nosmallerclassbuild}.
\begin{proof}[Proof of Theorem~\ref{thm:nosmallerclass}]
We assume that $\CC^M \supsetneq \CC^\omega$, since $\CC^\omega$ is the smallest Denjoy-Carleman class. We first prove the case $\F = \C$. Let $f$ be the function of Proposition~\ref{prop:nosmallerclassbuild} for the class $\CC^M$. Theorem~\ref{thm:nosmallerclass}\satref{sat:nosmallerclass:CM}. To prove Theorem~\ref{thm:nosmallerclass}\satref{sat:nosmallerclass:notCN}, note that if $U \subseteq (-1,1)$ is open, and $f \in \CC^N(U)$, then, for any $x \in U$, there is some open neighbourhood $V$ of $x$ contained in $U$ and constants $A,B>0$ such that
\[|f^{(j)}(x)| \leq AB^jj!N_j.\]
In particular, if $x$ is a dyadic rational in $V$, then, for all but finitely many $j$,
\[\recip{2}\recip{3^j}j!M_j \leq |f^{(j)}(x)| \leq AB^jj!N_j\]
which then implies that $\CC^M(U) \subseteq \CC^N(U)$.

Now consider $\F = \R$, and let $f$ be as above. For each dyadic rational $x \in (-1,1)$, and each $j$ large enough, either 
\[|\Re(f)^{(j)}(x)| \geq \recip{4}\recip{3^j}j!M_j \text{ or }|\Im(f)^{(j)}(x)| \geq \recip{4}\recip{3^j}j!M_j.\]
Set $g := \Re(f) + \Im(f)$. We show that $g$ satisfies the required properties. Clearly $g$ satisfies Theorem~\ref{thm:nosmallerclass}\satref{sat:nosmallerclass:CM}. For each dyadic rational in $x \in (-1,1)$ and for $j$ large enough, either
\[|g^{(j)}(x)| \geq |\Re(f)^{(j)}(x)| - |\Im(f)^{(j)}(x)| \geq \frac{2}{3}|\Re(f)^{(j)}(x)| = \recip{6}\recip{3^j}j!M_j\]
or
\[|g^{(j)}(x)| \geq |\Im(f)^{(j)}(x)| - |\Re(f)^{(j)}(x)| \geq \frac{2}{3}|\Im(f)^{(j)}(x)| = \recip{6}\recip{3^j}j!M_j.\]
So $g$ satisfies Theorem~\ref{thm:nosmallerclass}\satref{sat:nosmallerclass:notCN} for the same reason above as $f$ does.\end{proof}

Now we will prove Proposition~\ref{prop:nosmallerclassbuild}.
\begin{proof}[Proof of Proposition~\ref{prop:nosmallerclassbuild}]\phantomsection\label{pf:nosmallerclassbuild}
For any real number $\alpha > 0$, define
\[\phi(\alpha) := \sup_{\ell \geq 0} \frac{\alpha^{\ell+1}}{M_\ell} \text{  and  } m_n := \frac{M_{n+1}}{M_n}.\]
Recall that we are assuming that the sequence $M$ is logarithmically convex, i.e. the sequence $m_n$ is non-decreasing.
Since $\CC^M((-1,1)) \neq \CC^\omega((-1,1))$, $\phi(\alpha) < \infty$, for all $\alpha$. Furthermore,
\begin{equation}\label{eq:inACQ}M_n = \frac{m_n^{n+1}}{\phi(m_n)}\ .\end{equation}
A proof of \eqref{eq:inACQ} can be found in \cite[\S5\,, step 1]{ACQ}, but is repeated here for convenience.

By definition, it is required to prove that
\[\frac{m_n^{n+1}}{M_n} = \sup_{\ell \geq 0} \frac{m_n^{\ell+1}}{M_\ell}.\]
Indeed, if $\ell < n$, then
\[\frac{m_n^{\ell+1}}{M_\ell} \leq \frac{m_n^{\ell+1}}{M_\ell}\frac{m_n}{m_\ell} = \frac{m_n^{\ell+2}}{M_{\ell+1}}\]
and if $\ell > n$, then
\[\frac{m_n^{\ell+1}}{M_\ell} = \frac{m_n^{\ell}}{M_{\ell-1}}\frac{m_n}{m_\ell} \leq \frac{m_n^{\ell}}{M_{\ell-1}}.\]
The sequence $\frac{m_n^{\ell+1}}{M_\ell}$ is therefore non-decreasing for $\ell < n$ and non-increasing for $\ell > n$, and thus attains its supremum at $\frac{m_n^{n+1}}{M_n}$.

Now choose a non-decreasing sequence of integers $b_n$ satisfying:
\begin{enumerate}[label=(\roman*)]
\item $b_n \leq \min(m_n,2^n)$, for all $n$;
\item for all $n$, there is an integer $k_n$ such that $b_n = 2^{k_n}$
\item for all $k$, there is an integer $n_k$ such that $b_{n_k} = 2^k$;
\item $b_1 = 1$.
\end{enumerate}
For example, we can define the sequence $(b_n)$ recursively by $b_1 = 1$, and for all $n \geq 1$,
\[b_{n+1} := \begin{cases}
        b_n & \text{ if } 2b_n > m_{n+1} \\
	2b_n & \text{ if } 2b_n \leq m_{n+1}
        \end{cases}
\]
Then define $f$ by
\begin{equation}\label{eq:firstfunc}f(x) := \sum_{k=1}^\infty \recip{3^k\phi(m_k)} \sum_{a = -b_k}^{b_k} \recip{\left(x-\left(\frac{a}{b_k}+\frac{i}{m_k}\right)\right)}.\end{equation}
It helpful to picture the poles on the complex plane in \eqref{eq:firstfunc} both as coming in rows of height $\recip{m_k}$ and as columns lying above dyadic rationals in $(-1,1)$.

We will verify that $f$ satisfies the required properties.

First we will prove that \eqref{eq:firstfunc} converges uniformly on $(-1,1)$ together with its derivatives of every order. Then, $f \in \CC^\infty((-1,1))$ and we can differentiate \eqref{eq:firstfunc} term-by-term. Note that for any $s,t \in \R,$ $|s-it| \geq |t|$, and that, by the definition of $\phi$,
\[\frac{m_k^{j+1}}{\phi(m_k)} \leq M_j,\text{ for all } k,j.\]
We have the following estimates on the $j^\text{th}$ derivative of a general term in \eqref{eq:firstfunc}:
\begin{align*}
&\qquad \left|\left(\sum_{a = -b_k}^{b_k} \recip{\left(x-\left(\frac{a}{b_k}+\frac{i}{m_k}\right)\right)}\right)^{\hspace{-1ex}(j)}\right|\\
&= j!\left|\recip{3^k\phi(m_k)} \sum_{a = -b_k}^{b_k} \recip{\left(x-\left(\frac{a}{b_k}+\frac{i}{m_k}\right)\right)^{j+1}}\right|\\
&\leq j!\recip{3^k\phi(m_k)} \sum_{a = -b_k}^{b_k}\recip{\left|\frac{i}{m_k}\right|^{j+1}}\\
&= j! \frac{m_k^{j+1}}{3^k\phi(m_k)}(2b_k+1)\\
&\leq j!M_j \frac{2b_k+1}{3^k} \leq M_jj!\frac{2^{k+1}+1}{3^k} = \frac{9}{2}j!M_j.
\end{align*}
Since
\[\sum_{k=1}^\infty M_jj!\frac{2^{k+1}+1}{3^k} = \frac{9}{2}j!M_j,\]
the series in \eqref{eq:firstfunc} converges absolutely and uniformly on $(-1,1)$ by the M-test.

Differentiating term-by-term, the above computation gives the upper bounds \satref{sat:nosmallerclassbuild:upper} on the derivatives of $f$.

We prove the lower bounds \satref{sat:nosmallerclassbuild:lower} on the derivatives of $f$ at dyadic rationals, at the same time as \satref{sat:nosmallerclassbuild:realim}. The idea is, for any given dyadic rational $t = \frac{p}{2^q} \in (-1,1)$, to look at those summands in \eqref{eq:firstfunc} which have poles on vertical lines lying above $t$. Since by construction there are only finitely many rows of poles not containing a pole lying above $t$, the sum of these summands is analytic when restricted to $(-1,1)$, and thus will not affect the estimate. For the remaining rows, the sum over the $j^{\text{th}}$ derivatives of summands with poles \emph{not} lying above $t$ is a multiple of the sum over the $j^{\text{th}}$ derivatives of the summands with poles \emph{which do} lie above $t$, and this multiple can be made arbitrarily small for large $j$. So, as long as the sum of the $j^{\text{th}}$ derivatives of the summands with poles lying above $t$ is large, the $j^{\text{th}}$ derivative of $f$ at $t$ will be large too.

To show this explicitly, fix some dyadic rational $t = \frac{p}{2^q} \in (-1,1)$. Then, for some large $K = K_t$, $b_k \geq 2^q$ for all $k \geq K$. Thus, we can write
\begin{align*}
f(x) =  &\sum_{k<K} \recip{3^k\phi(m_k)} \sum_{a = -b_k}^{b_k} \recip{\left(x-\left(\frac{a}{b_k}+\frac{i}{m_k}\right)\right)} \\
+&\sum_{k\geq K} \recip{3^k\phi(m_k)} \sum_{a = -b_k}^{b_k} \recip{\left(x-\left(\frac{a}{b_k}+\frac{i}{m_k}\right)\right)}.
\end{align*}
Call the first sum $f_1(x)$ and the second sum $f_2(x)$. $f_1$ is clearly holomorphic in an open neighbourhood of $(-1,1)$ in $\C$, and is thus in particular analytic on $(-1,1)$. So, there are $E,F > 0$ such that $|f_1^{(j)}(t)| \leq EF^jj!$, for all $j \geq 0$. Since we can differentiate the series for $f(x)$ term-by-term, we can also differentiate the series for $f_2(x)$ term-by-term. In particular,
\begin{align*}
&\quad \ f_2^{(j)}(t)/j! = \sum_{k\geq K}^\infty \frac{(-1)^j}{3^k\phi(m_k)} \sum_{a = -b_k}^{b_k} \recip{\left(t-\left(\frac{a}{b_k}+\frac{i}{m_k}\right)\right)^{j+1}}\\
&= \sum_{k\geq K} \frac{-1}{3^k\phi(m_k)} \recip{\left(\frac{i}{m_k}\right)^{j+1}} + \sum_{k\geq K} \frac{(-1)^j}{3^k\phi(m_k)} \sum_{\substack{-b_k \leq a \leq b_k \\ a/b_k \neq t}} \recip{\left(t-\left(\frac{a}{b_k}+\frac{i}{m_k}\right)\right)^{j+1}}.
\end{align*}

Call the first of these sums $S_{1,j}$, and the second $S_{2,j}$. Clearly, for $j \geq K$,
\begin{align*}
|S_{1,j}| = \sum_{k\geq K} \frac{m_k^{j+1}}{\phi(m_k)3^k} \geq \recip{3^j}\frac{m_j^{j+1}}{\phi(m_j)} = \recip{3^j}M_j,
\end{align*}
by \eqref{eq:inACQ}. If $j$ is odd, then $|\Re(S_{1,j})| = |S_{1,j}|$, and $|\Im(S_{1,j})| = 0$, with the roles of the real and imaginary parts reversed if $j$ is even.

Remembering that $b_k \leq m_k$ for all $k$, and that $b_k$ is a power of $2$ bigger than $2^q$ for all $k \geq K$ (and hence $tb_k-a \in \Z$ for all $a \in \Z$) we also have that
\begin{align*}
|S_{2,j}| &\leq \sum_{k\geq K} \recip{3^k\phi(m_k)} \sum_{\substack{-b_k \leq a \leq b_k \\ a/b_k \neq t}} \recip{\left|\left(t-\frac{a}{b_k}\right)-\left(\frac{i}{m_k}\right)\right|^{j+1}}\\
&= \sum_{k\geq K} \recip{3^k\phi(m_k)} \sum_{\substack{-b_k \leq a \leq b_k \\ a/b_k \neq t}} \frac{m_k^{j+1}}{\left(\frac{m_k^2}{b_k^2}(tb_k-a))^2 + 1\right)^{\frac{j+1}{2}}}\\
&\leq \sum_{k\geq K} \frac{m_k^{j+1}}{3^k\phi(m_k)}\sum_{\substack{-\infty < n < \infty \\ n \neq 0}} \frac{1}{\left(n^2 + 1\right)^{\frac{j+1}{2}}}\\
&= \left(\sum_{\substack{-\infty < n < \infty \\ n \neq 0}} \frac{1}{\left(n^2 + 1\right)^{\frac{j+1}{2}}}\right)\left(\sum_{k\geq K} \frac{m_k^{j+1}}{3^k\phi(m_k)}\right).
\end{align*}
The second factor is just $|S_{1,j}|$. Call the first factor $C_{j}$. Then, for $j \geq K$ odd,
\begin{align*}
|\Re(f)^{(j)}(t)/j!| &\geq |\Re(S_{1,j})|-|\Re(S_{2,j})|-|\Re(f_1)^{(j)}(t)/j!|\\
&\geq |S_{1,j}|-|S_{2,j}|-|f_1^{(j)}(t)/j!| \geq (1-C_{j})|S_{1,j}|-EF^j
\end{align*}
and
\begin{align*}
|\Im(f)^{(j)}(t)/j!| &\leq|\Im(S_{1,j})|+|\Im(S_{2,j})|+|\Im(f_1^{(j)})(t)/j!|\\
&\leq |S_{2,j}|+|f_1^{(j)}(t)/j!| \leq C_j|S_{1,j}|+EF^j.
\end{align*}
with the roles of the real and imaginary parts reversed if $j \geq K$ is even.

Since for large enough $j$, $EF^j < \frac{3}{32}\recip{3^j}M_j < \recip{8}\recip{3^j}M_j$ (since $M_j$ grows more quickly than any exponential), if for large enough $j$, $C_j < 3/48 < 1/8$, we would have for large odd $j$
\begin{align*}
\recip{3}|\Re(f)^{(j)}(t)/j!|-|\Im(f)^{(j)}(t)/j!| &\geq \recip{3}((1-C_{j})|S_{1,j}|-EF^j)-(C_j|S_{1,j}|+EF^j)\\
&= \left(\recip{3}-\frac{4}{3}C_j\right)|S_{1,j}|-4/3EF^j\\
&\geq \recip{4}\recip{3^j}M_j-4/3EF^j \geq \recip{8}\recip{3^j}M_j > 0,
\end{align*}
so that both
\[|\Im(f)^{(j)}(t)| \leq \recip{3}|\Re(f)^{(j)}(t)|\]
and
\begin{align*}
|f^{(j)}(t)| &\geq |\Re(f)^{(j)}(t)| - |\Im(f)^{(j)}(t)| \geq \frac{2}{3}|\Re(f)^{(j)}(t)| \\
&\geq  \frac{2}{3}j!\left((1-C_j)|S_{1,j}|-EF^j\right) \geq \frac{2}{3}j!\left(\frac{7}{8}\recip{3^j}M_j-\recip{8}\recip{3^j}M_j\right) = \recip{2}\recip{3^j}j!M_j.\end{align*}
If $j$ is even, then the roles of the real part and imaginary part are reversed. So  \satref{sat:nosmallerclassbuild:lower} and \satref{sat:nosmallerclassbuild:realim} would follow provided that $C_{j} \to 0$ as $j \to \infty$. Indeed
\begin{align*}
C_j &= 2\sum_{n=1}^\infty\recip{(n^2+1)^{\frac{j+1}{2}}}\\ 
&\leq 2\left(\recip{\sqrt{2}^{j+1}} + \sum_{n=2}^\infty \recip{n^{j+1}}\right)\\
&\leq 2\left(\recip{\sqrt{2}^{j+1}} + \int_1^\infty \recip{x^{j+1}}\text{d}x\right) \leq 2\left(\recip{\sqrt{2}^{j+1}} + \recip{j}\right) \to 0 \text{ as } j \to \infty,
\end{align*}
as desired.
\end{proof}

\section{A function formally in a given Denjoy-Carleman class at every point, yet not in the class}\label{sec:formalnotCM}
The idea for the construction of such a function will be to build it as a series of functions $f_k$ whose $k^{\text{th}}$ derivatives at points $a_k$ are large, where $(a_k)$ is a sequence tending to $0$, and whose derivatives at points other than $a_k$ are sufficiently nice. The following proposition is in some sense a simplified version of the example constructed in Theorem~\ref{prop:nosmallerclassbuild}, and will provide the building blocks of our example.

\begin{prop}\label{prop:elementarybuildingblocks}For any non-analytic Denjoy-Carleman class $\CC^M$, there exists $f \in \CC^\infty(\R)$ satisfying\emph{:}
\begin{enumerate}[label = \emph{(\arabic*)}]
\item \label{sat:elementarybuildingblocks:CMupper}for all $j \geq 0$, and all $x \in \R$, $|f^{(j)}(x)| \leq j!M_j$\emph{;}
\item \label{sat:elementarybuildingblocks:Analupper}for all $j \geq 0$, and all $x \neq 0$, $|f^{(j)}(x)| \leq j!|x|^{-(j+1)}$\emph{;}
\item \label{sat:elementarybuildingblocks:lower}for all $j \geq 1$, $|f^{(j)}(0)| \geq \recip{2^j}j!M_j$.
\end{enumerate}
\end{prop}
\begin{proof}
Let $m_n := M_{n+1}/M_n$, and let
\[\phi(\alpha) := \sup_{\ell \geq 0} \frac{\alpha^{\ell+1}}{M_\ell},\]
as in the \hyperref[pf:nosmallerclassbuild]{proof of Proposition}~\ref{prop:nosmallerclassbuild} (recalling again the hypothesis of logarithmic convexity). Define
\begin{equation}\label{eq:secfunc}f(x) = \sum_{k=1}^\infty \recip{2^k\phi(m_k)\left(x-\frac{i}{m_k}\right)}.\end{equation}

We will prove that $f$ satisfies all the required properties.

First we will show that \eqref{eq:secfunc} converges uniformly on $\R$ together with its derivatives of every order. Then, $f \in \CC^\infty(\R)$ and we can differentiate term-by-term. Indeed, we have the following estimates on the $j^\text{th}$ derivative of a general term of the series in \eqref{eq:secfunc}:
\begin{align*}
&\left|\left(\recip{2^k\phi(m_k)\left(x-\frac{i}{m_k}\right)}\right)^{\hspace{-1ex}(j)}\right| =  j!\left|\recip{2^k\phi(m_k)\left(x-\frac{i}{m_k}\right)^{j+1}}\right|\\
&\leq j!\recip{2^k\phi(m_k)\left|\frac{i}{m_k}\right|^{j+1}}
= j!\recip{2^k}\frac{m_k^{j+1}}{\phi(m_k)} \leq j!M_j\recip{2^k}.
\end{align*}
Since
\[\sum_{k=1}^\infty j!M_j\recip{2^k} = j!M_j,\]
The series in \eqref{eq:secfunc} converges absolutely and uniformly on $\R$ by the M-test.

Differentiating term-by-term, the above computation gives the upper bounds \satref{sat:elementarybuildingblocks:CMupper} on the derivatives of $f$.

We next prove \satref{sat:elementarybuildingblocks:Analupper}. Note that for all $k$, $\phi(m_k) \geq 1$. Indeed,
\begin{align*}
\phi(m_k) &= \sup_{\ell \geq 0} \frac{m_k^{\ell+1}}{M_\ell} \geq \frac{m_k^{k+1}}{M_k} \geq \frac{m_km_{k-1}\cdots m_1m_0}{M_k}\\
&= \recip{M_k}\frac{M_{k+1}}{M_k}\frac{M_{k}}{M_{k-1}}\cdots\frac{M_2}{M_1}\frac{M_1}{M_0} = \frac{M_{k+1}}{M_k}\recip{M_0} \geq 1.
\end{align*}
So, for all $x \neq 0$,
\begin{align*}
|f^{(j)}(x)| &= j!\left|\sum_{k=1}^\infty\recip{2^k\phi(m_k)\left(x-\frac{i}{m_k}\right)^{j+1}}\right|\\
&\leq j!\sum_{k=1}^\infty\recip{2^k\phi(m_k)|x|^{j+1}}= |x|^{-(j+1)}j!\sum_{k=1}^\infty\recip{2^k\phi(m_k)} \leq j!|x|^{-(j+1)}.
\end{align*}

To prove the lower bounds \satref{sat:elementarybuildingblocks:lower} on the derivatives at 0, note that for $j \geq 1$,
\begin{align*}
|f^{(j)}(0)| &= j!\left|\sum_{k=1}^\infty\recip{2^k\phi(m_k)\left(-\frac{i}{m_k}\right)^{j+1}}\right| = j!\sum_{k=1}^\infty\frac{m_k^{j+1}}{2^k\phi(m_k)} \geq \recip{2^j}\frac{m_j^{j+1}}{\phi(m_j)}j! = \recip{2^j}j!M_j.
\end{align*}
\end{proof}

The proofs of Theorem~\ref{thm:formaleverywherenot} and Theorem~\ref{thm:masterthm} will be somewhat simplified by introducing \emph{strictly} logarithmically convex weight sequences $M$ for our Denjoy-Carleman classes.
\begin{defn} A sequence $M = (M_n)_{n=0}^\infty$ is called \emph{strictly} logarithmically-convex if the ratios $M_{n+1}/M_n$ form a strictly-increasing sequence.\end{defn} Notice that strict logarithmic convexity also implies that the sequence $M_n^{1/n}$ is strictly increasing.

\begin{lem}\label{lem:strictlogconvex}Let $\mathcal C^M$ denote a non-analytic Denjoy-Carleman class. Then, there exists a non-decreasing strictly-logarithmically convex sequence $\widetilde{M}$ such that $\mathcal C^M = \mathcal C^{\widetilde{M}}.$\end{lem}
\begin{proof}
For $n \geq 0$, set $m_n = M_{n+1}/M_n$. Partition $\N$ into a union of disjoint intervals $S_k$ on which $m_n$ is constant, i.e.
\[\N = \bigcup_{k=0}^\infty S_k,\]
where $S_k = \{n_k,n_k+1,\ldots,n_k+\ell_k-1\}$,
$m_n = m_{n'}$ for all $n,n' \in S_k$, and $m_{n_{k+1}-1} < m_{n_{k+1}}$. Notice that each $S_k$ really is finite since $\mathcal C^M$ is non-analytic, and that $\#S_k = \ell_k$. We define a sequence $(a_n)_0^\infty$ of real numbers as follows: set
\[A := \min\left(2,\frac{m_{n_{k+1}}}{m_{n_{k+1}-1}}\right)\]
and then if $n=n_k+i \in S_k$, 
\[a_{n} := A^{i/\ell_k}.\]
Notice that since $n_{k+1} \in S_{k+1}$ but $n_{k+1}-1 \in S_k$ $\frac{m_{n_{k+1}}}{m_{n_{k+1}-1}}, A > 1$ and also that $1 \leq a_n \leq 2$ for all $n$.
Define $\widetilde{M}_0 = M_0 = 1$
and
\[\widetilde{M}_n = M_n\prod_{k=0}^{n-1}a_k\]
for $n \geq 1$. It is easy to verify that $\widetilde{M}$ is non-decreasing, strictly logarithmically-convex, and that $\CC^M = \CC^{\widetilde{M}}$. 
\end{proof}

\begin{proof}[Proof of Theorem~\ref{thm:formaleverywherenot}]\phantomsection\label{pf:formaleverywherenot}
The case $\CC^M = \CC^\omega$ is easy; the function $f(x) = e^{-1/x^2}$ satisfies all the necessary propeties. Assume from now on that $\CC^M \neq \CC^\omega$. In light of Lemma~\ref{lem:strictlogconvex}, we might as well assume that $M$ is strictly logarithmically convex. The function in the construction below is complex-valued. The case $\F=\R$ follows from the case $\F=\C$ by considering real and imaginary parts. For the case $\F=\C$, the idea is to construct $f$ as an infinite sum of functions described in Proposition~\ref{prop:elementarybuildingblocks}, but shifted so that the points at which we have a lower bound on the derivatives, analogous to those of Proposition~\ref{prop:elementarybuildingblocks}\satref{sat:elementarybuildingblocks:lower}, are on a sequence tending to $0$. Consider the sequence $(M_n^{1/n})_{n=1}^\infty$. Since $\CC^M$ is not analytic, $M_n^{1/n} \to \infty$. Set $b_n = M_n^{1/n}$. Note that the terms $b_n$ are strictly increasing (and in particular distinct) since $M$ is strictly logarithmically convex.
Then, define $a_n := \recip{\sqrt{b_n}}$ for all $n$, so that $a_n \to 0$.

We also define a family of non-decreasing, logarithmically-convex sequences indexed by $k$ ($k \in \Z$, $k \geq 1$), $M^k = (M^k_n)_{n=0}^\infty$, with $M^k_0 = 1$ by
\[M^k_n := \begin{cases} 
	    1 & \text{ if } k > n \\
            c_k^{2n-2k+1}M_n & \text{ if } k \leq n
            \end{cases}\]
for all $k \geq 1$, where $c_k\geq M_k$ are large constants to be determined later, but which will depend only on the sequences $(a_n)$ and $(M_n)$.

\phantomsection\label{pf:formaleverywherenot:sameclass}Notice that $\CC^M = \CC^{M^k}$, for all $k \geq 1$. Let $h_k$ be the function given by Proposition~\ref{prop:elementarybuildingblocks} applied to the sequence $M^k$, and set $f_k(x) = h_k(x-a_k)$, for all $k$. Then the $f_k \in \CC^\infty(\R)$ and satisfy:
\begin{enumerate}[label = {(\roman*)}]
\item for all $j \geq 0$ and all $x \in \R$, $|f_k^{(j)}(x)| \leq j!M^k_j$;
\item for all $j \geq 0$ and for all $x \neq a_k$, $|f_k^{(j)}(x)| \leq |x-a_k|^{-(j+1)}j!$;
\item for all $j \geq 1$ $|f_k^{(j)}(a_k)| \geq \recip{2^j}j!M^k_j$.
\end{enumerate}
Define
\begin{equation}\label{eq:thirdfunc}f(x) := \sum_{k=1}^\infty \recip{2^k}f_k(x).\end{equation}

We will verify that $f$ satisfies all of the necessary properties. 

\label{pf:formaleverywherenot:smooth}First we prove that \eqref{eq:thirdfunc} converges uniformly on $\R$ together with its derivatives of every order. Then, $f \in \CC^\infty(\R)$ and we can differentiate term-by-term. following estimates on the $j^\text{th}$ derivative of a general term of the series in \eqref{eq:thirdfunc}"'
\begin{align*}
\left|\recip{2^k}f_k^{(j)}(x)\right| &\leq \recip{2^k}M^k_jj!= 
\begin{cases}\frac{j!M^k_j}{2^k}\quad&\text{ if } k \leq j\\
\frac{j!}{2^k}\quad&\text{ otherwise.}\end{cases}\end{align*}
Since the sum
\[\sum_{k=1}^j \frac{j!M^k_j}{2^k} + \sum_{k=j+1}^\infty \frac{j!}{2^k} < \infty,\]
the sum converges absolutely and uniformly on $\R$ by the M-test.

\label{pf:formaleverywherenot:CM}To prove \satref{sat:formaleverywherenot:CM}, we show that for each $x \neq 0$, there is some neighbourhood $U$ containing $x$ and constants $A,B$ such that, for all $j$ and all $y \in U$,
\[|f^{(j)}(y)| \leq AB^jj!M_j.\]
We distinguish two cases: $x \neq a_n$ for all $n$, and $x = a_n$, for some $n$. In the first case, there is a neighbourhood $U$ of $x$ and a $\delta > 0$ such that $\inf_k |y-a_k| > \delta$ for all $y \in U$. Then we see that, for $y \in U$ and $j \geq 0$,
\begin{align*}
|f^{(j)}(y)| &\leq j!\sum_{k=1}^\infty \recip{2^k}|x-a_k|^{-(j+1)}\\
&\leq j!\sum_{k=1}^\infty \recip{2^k}\left(\recip{\delta}\right)^{j+1} \hspace{-6pt} = \recip{\delta}\left(\recip{\delta}\right)^{j}j! \leq \recip{\delta}\left(\recip{\delta}\right)^{j}j!M_j.
\end{align*}

In the second case, suppose $x = a_n$. Then there is a neighbourhood $U$ of $a_n$ and $\delta = \delta_n > 0$ such that $\inf_{k \neq n} |y-a_k| > \delta$ for all $y \in U$. Let $A = \max(\delta^{-1}, 1)$. We see that, for $y \in U$ and $j \geq 0$,
\begin{align*}
|f^{(j)}(y)| &\leq j!\sum_{k\neq n} \recip{2^k}|x-a_k|^{-(j+1)} + j!\recip{2^n}M^n_j\\
&\leq j!\sum_{k=1}^\infty \recip{2^k}\left(\recip{\delta}\right)^{j+1} \hspace{-6pt}+ j!c_n^{2j+1}M_j
= \recip{\delta}\left(\recip{\delta}\right)^{j}j!+j!c_n^{2j+1}M_j\\ \leq (2Ac_n)(c_n^2A)^jj!M_j.\\
\end{align*}

\label{pf:formaleverywherenot:formal}Showing \satref{sat:formaleverywherenot:formal} is an easy computation. Recall that, by the logarithmic convexity of $M$, for any positive integers $j,k$ with $k \leq j$, $M_k^{1/k} \leq M_j^{1/j}$. So, for $j \geq 1$,
\begin{align*}
|f^{(j)}(0)| &\leq j!\sum_{k=1}^j \recip{2^k}|a_k|^{-(j+1)} + j!\sum_{k=j+1}^\infty \recip{2^k}M^k_j \\
&\leq j!\sum_{k=1}^j\sqrt{b_k}^{2j} + j!\sum_{k=j+1}^\infty \recip{2^k} = j!\sum_{k=1}^j M_k^{j/k} + \frac{j!}{2^j} \leq 2e^jj!M_j.
\end{align*}

\label{pf:formaleverywherenot:notCM}In order to show \satref{sat:formaleverywherenot:notCM}, we will need to pick appropriate $c_n$. Note that for all $n\geq 1$
\begin{align}
|f^{(n)}(a_n)| &\geq \recip{2^n}\recip{2^n}M^n_nn!-n!\sum_{k \neq n}\recip{2^k}|a_n-a_k|^{-(n+1)} \nonumber\\
\label{eq:lowerboundestimateone}&= \recip{2^n}\recip{2^n}c_nM_nn! - n!\sum_{k \neq n}\recip{2^k}|a_n-a_k|^{-(n+1)}.
\end{align}
Since
\[n!\sum_{k \neq n}\recip{2^k}|a_n-a_k|^{-(n+1)} < n!\left(\inf_{n\neq k} |a_n-a_k|\right)^{-(n+1)} < \infty,\]
we can choose $c_n \geq M_n$ large so that \eqref{eq:lowerboundestimateone} is bigger than $n^nn!M_n$, and hence
\begin{align*}
|f^{(n)}(a_n)| \geq n^nn!M_n.
\end{align*}
So, if $f \in \CC^M(\R)$, then there would be some $\epsilon > 0$ and constants $A,B > 0$ such that for $|x| < \epsilon$
\[|f^{(n)}(x)| \leq AB^nn!M_n.\]
In particular, for all but finitely many $n$, $|a_n| < \epsilon$ and
\[n^nn!M_n \leq |f^{(n)}(a_n)| \leq AB^nn!M_n.\]
which is impossible, since $n^n$ grows more quickly than any exponential.
\end{proof}

\begin{prop}\label{prop:formalmeansalmostCM} Let $\CC^M$ be any Denjoy-Carleman class, $U \subseteq \R^p$ open (for $p \geq 1$), and suppose $f \in \CC^\infty(U,\F)$. Then, if $f \in \mathcal F^M(x,\F)$, for each $x \in U$ there exists an open dense subset $V$ of $U$ such that $f \in \CC^M(V,\F)$.\end{prop}
\begin{proof}
It suffices to prove that for each non-empty open $W_1 \subseteq U$, there exists a non-empty open $W_2 \subseteq W_1$ such that $f \in \CC^M(W_2,\F)$. So, suppose a non-empty open $W_1 \subseteq U$ is given. Let $W' \subseteq W_1$ be open, bounded, with its closure contained inside $W_1$. Let $A$ be an upper bound of $f$ on $W'$. Set $A' = \max(A,1)$, and for each $B > 0$ set
\[S_B := \{x \in W' \: |D^\alpha f(x)| \leq A'B^{|\alpha|}|\alpha|!M_{|\alpha|} \text{ for all } \alpha\in\N^p\}.\]
By assumption, since for each $x \in W'$, $f \in \mathcal F^M(x,\F)$, there are $P_x,Q_x > 0$ such that
\[|D^\alpha f(x)| \leq P_xQ_x^{|\alpha|}|\alpha|!M_{|\alpha|}\]
for all $\alpha\in\N^p$. Considering the cases $P_x/A' \leq 1$ and $P_x/A' > 1$ separately, it is easy to see that for each $x \in W'$, there is some $B>0$ such that $x \in S_B$. It follows that
\[W' = \bigcup_{N=1}^\infty S_N.\]
Since for each $\alpha$, $D^\alpha f$ is continuous, each $S_N$ is closed (with respect to the subspace topology on $W'$). Since $W'$ is locally compact and Hausdorff, the Baire category theorem provides at least one $N_0$ such that $S_{N_0}$ has non-empty interior (with respect to the subspace topology on $W'$). Let $W_2$ be the interior of $S_{N_0}$. By definition $f \in \CC^M(W_2,\F)$, and $W_2 \subseteq W_1$ is open, as desired.
\end{proof}

\section{A smooth function which is quasianalytic on every curve (of a given quasianalytic Denjoy-Carleman class), yet not in the class}

The idea for constructing this function is similar in spirit to the idea for the function constructed in \S\ref{sec:formalnotCM}. The idea is to construct $f$ as a series of functions $f_k$ whose $(2k)^\text{th}$ derivatives at points $a_k$ is large, where $(a_k)$ is a sequence tending to $0$ on some flat curve, and whose derivatives at points other than $a_k$ is sufficiently nice. Since there are no quasianalytic flat curves, this will imply that the function will be quasianalytic on each quasianalytic curve, but will not be quasianalytic.

We first give an analogue of Proposition~\ref{prop:elementarybuildingblocks} for dimension $>1$; this is Proposition~\ref{prop:ndbuildingblocks}, below. The proof of the latter uses the following lemma, which provides a way of passing a function in one variable with given derivative bounds to a function in many variables with similar derivative bounds.

\begin{lem}\label{lem:passingtohigherdimensions} Let $p \geq 2$, and let $g \in \CC^\infty(\R)$ denote a function such that
 \[|g^{(j)}(t)| \leq j!C_{t,j},\]
where $C_{t,j}$ is a non-decreasing sequence for each $t \in \R$. Set
\[f(x) := g(||x||^2) = g(x_1^2+\cdots+x_p^2).\]
Then $f \in \CC^\infty(\R^p)$ and\emph{:}
\begin{enumerate}[label = \emph{(\arabic*)}]
\item \label{sat:passingtohigherdimensions:upper}for all $\alpha \in \N^p$, \[|D^\alpha f(x)| \leq (B(||x||+1))^{|\alpha|}|\alpha|!C_{||x||^2,|\alpha|}\emph{;}\]
\item \label{sat:passingtohigherdimensions:lower}for all $1\leq i \leq p$ and $n \geq 0$, \[\frac{\partial^{2n}f}{\partial x_i^{2n}} (0) = g^{(n)}(0)\frac{(2n)!}{n!}\emph{;}\]
\end{enumerate}
where $B$ depends only on $p$ (not on $g$, $\alpha$, or $x$). 
\end{lem}
\begin{proof}
By a multivariate version of Fa\`{a} di Bruno's formula (see, for instance, \cite[Prop.\,4.3]{BM}) applied to $g(||x||^2)$,
\begin{equation}\label{eq:faadibrunomod}D^\alpha f(x) = \alpha!\sum \recip{k_{1,1}!k_{1,2}!\cdots k_{p,1}!k_{p,2}!}g^{(n)}(||x||^2)\prod_{j=1}^p(2x_j)^{k_{j,1}},\end{equation}
where $n = k_{1,1} + k_{1,2} + \cdots + k_{p,1}+k_{p,2}$ and the sum is taken over all $2p$--tuples of non-negative integers $(k_{1,1},k_{1,2},\ldots,k_{p,1},k_{p,2})$ such that
\begin{equation}\label{eq:sumright}\alpha = (\alpha_1,\ldots,\alpha_p) = (k_{1,1}+2k_{1,2},\ldots,k_{p,1}+2k_{p,2}).\end{equation}
Since $n = k_{1,1}+\cdots+k_{p,2} \leq \alpha_1+\cdots+\alpha_p = |\alpha|$ whenever $k_{i,j}$ satisfy \eqref{eq:sumright} ($1\leq i\leq p, j=1,2)$, we see that
\begin{align}
|D^\alpha f(x)| &\leq \alpha!\sum \recip{k_{1,1}!k_{1,2}!\cdots k_{p,1}!k_{p,2}!}|g^{(n)}(||x||^2)|\prod_{j=1}^p(2|x_j|)^{k_{j,1}} \nonumber \\
&\leq \alpha!\sum \recip{k_{1,1}!k_{1,2}!\cdots k_{p,1}!k_{p,2}!}n!C_{||x||^2,|\alpha|}2^{|\alpha|}(||x||+1)^{|\alpha|} \nonumber \\
\label{eq:step1}&\leq (2(||x||+1))^{|\alpha|}C_{||x||^2,|\alpha|}|\alpha|!\sum \frac{n!}{k_{1,1}!\cdots k_{p,2}!},
\end{align}
where the summation is as in \eqref{eq:faadibrunomod}.
By the multinomial theorem,
\[\frac{n!}{k_{1,1}!\cdots k_{p,2}!} \leq \sum_{\ell_1+\cdots+\ell_{2p}=n} \frac{n}{\ell_1\cdots\ell_{2p}} = (\underbrace{1+\cdots+1}_{2p \text{ 1's}})^n = (2p)^n \leq (2p)^{|\alpha|}.\]
Thus, from \eqref{eq:step1},
\[|D^\alpha f(x)| \leq (4p(||x||+1))^{|\alpha|}|\alpha!|C_{||x||^2,|\alpha|}\# S,\]
where $S$ is the set of all $2p$--tuples of non-negative integers $(k_{1,1},k_{1,2},\ldots,k_{p,1},k_{p,2})$ satisfying \eqref{eq:sumright}. Since, for each $i$, $\alpha$ and $k_{i,1}$ uniquely determine $k_{i,2}$, and there are at most $|\alpha|+1$ choices of $k_{i,1}$, $\#S \leq (|\alpha|+1)^p \leq (e^p)^{|\alpha|}$. So in all,
\[|D^\alpha f(x)| \leq (4pe^p(||x||+1))^{|\alpha|}|\alpha!|C_{||x||^2,|\alpha|},\] which is  \satref{sat:passingtohigherdimensions:upper}. \satref{sat:passingtohigherdimensions:lower} is obvious either again from Fa\`{a} di Bruno's formula, or by looking at the formal power series of $g$ at $0$.
\end{proof}
\newcommand{\Alp}{|\alpha|}
\begin{prop}\label{prop:ndbuildingblocks} For any $p \geq 2$ and any non-analytic Denjoy-Carleman class $\CC^M$, there exists $f \in \CC^M(\R^p)$ satisfying\emph{:}
\begin{enumerate}[label = \emph{(\arabic*)}]
\item \label{sat:ndbuildingblocks:upperCM}for any compact $K \subseteq \R^p$, and for all $\alpha \in \N^p$, $x \in K$, \[|D^{\alpha}f(x)| \leq (B(||K||+1))^{|\alpha|}|\alpha|!M_{|\alpha|}\emph{;}\]
\item \label{sat:ndbuildingblocks:upperanal}for any compact $K \subseteq \R^p$, and for all $\alpha \in \N^p$, $x \in K\setminus\{0\}$, \[|D^{\alpha}f(x)| \leq (B(||K||+1))^{|\alpha|}|\alpha|!||x||^{-2(|\alpha|+1)},\text{ if }||x|| \leq 1\emph{;}\]
\item \label{sat:ndbuildingblocks:uppernice}for any compact set $K \subseteq \R^p$, and for all $\alpha \in \N^p$, $x \in K\setminus\{0\}$, \[|D^{\alpha}f(x)| \leq (B(||K||+1))^{|\alpha|}|\alpha|!,\text{ if }||x|| \geq 1\emph{;}\]
\item \label{sat:ndbuildingblocks:lower}for all $n \geq 1$, \[\left|\frac{\partial^{2n}f}{\partial x_1^{2n}}(0)\right| \geq (2n)!M_n,\]
\end{enumerate}
where $B$ depends only on $p$ (as in Lemma~\ref{lem:passingtohigherdimensions}; $B$ and does not depend on $M$ or $K$).
\end{prop}
\begin{proof}
Apply Lemma~\ref{lem:passingtohigherdimensions} to Proposition~\ref{prop:elementarybuildingblocks}.
\end{proof}

Let $p \geq 2$. For any integer $m \geq 1$, and real number $a > 0$, we denote by $\mathcal S_{a,m}^p$ the set
\[\{x=(x_1,x_2,\ldots,x_p) \in \R^p \: x_1 \geq 0 \text{ and } x_2 \geq ax_1^m\}\]
and by $\mathcal Q^p$ the set
\[\{x=(x_1,x_2,\ldots,x_p) \in \R^p \: x_1 > 0 \text{ and } x_2 > 0\}.\]
The following lemma is elementary:
\begin{lem}\label{lemma:Sdistance}Let $p \geq 2$, $m \geq 1$ an integer, and $a > 0$ a real number. Let $\mathcal S = \mathcal S_{a,m}^p$. Then, for sufficiently small positive $t$,
\[\dist((t,e^{-\recip{t^2}},0,\ldots,0),\mathcal S) := \inf_{s \in \mathcal S}||(t,e^{-\recip{t^2}},0,\ldots,0)-s|| \geq e^{-\recip{t^2}}.\]
\end{lem}

\begin{proof}[Proof of Theorem~\ref{thm:masterthm}]
The proof of is very similar to \hyperref[pf:formaleverywherenot]{that of Theorem}~\ref{thm:formaleverywherenot}. The case $\F=\R$ follows immediately from the case $\F = \C$ by considering real and imaginary parts. For the case $\F = \C$, consider the sequence $(b_n)_{n=1}^\infty = (M_n^{1/n})_{n=1}^\infty$. Since $\CC^M \neq \CC^\omega$, $b_n \to \infty$. In light of Lemma~\ref{lem:strictlogconvex}, we might as well assume that the terms of $b_n$ are distinct. For $n \geq 1$, set
\[a_n := \left(\sqrt{\recip{\log{b_n^{1/4}}}},\recip{b_n^{1/4}},0,\ldots,0\right).\]
Then $a_n \to 0$, and $a_n \in \{(t,e^{-\recip{t^2}},0,\ldots,0) \: t > 0\}$.
Define a family of non-decreasing, logarithmically-convex sequences indexed by $k$ ($k \in \Z$, $k \geq 1$), $M^k = (M^k_n)_{n=0}^\infty$, with $M^k_0 = 1$ by
\[M^k_n := \begin{cases} 
	    1 & \text{ if } k > n \\
            c_k^{2n-2k+1}M_n & \text{ if } k \leq n
            \end{cases}\]
where $c_k\geq M_k$ are large constants to be determined later, but which will depend only on the sequences $(a_n)$ and $(M_n)$.

As in the \hyperref[pf:formaleverywherenot:sameclass]{proof of Theorem}~\ref{thm:formaleverywherenot}, $\CC^{M^k} = \CC^M$ for all $k$. Let $h_k$ be the function given by Proposition~\ref{prop:ndbuildingblocks} applied to the sequence $M^k$, and set $f_k(x) = h_k(x-a_k)$, for all $k$. Let $a = 1+\sup_{k\geq1}|a_k|$. Then the $f_k \in \CC^\infty(\R^p)$ and satisfy:
\begin{enumerate}[label = {(\roman*)}]
\item for any compact $K \subseteq \R^p$, and for all $\alpha \in \N^p$, $x \in K$, \[|D^{\alpha}f_k(x)| \leq (B(||K||+a))^{|\alpha|}|\alpha|!M^k_{|\alpha|}\text{;}\]
\item for any compact $K \subseteq \R^p$, and for all $\alpha \in \N^p$, $x \in K\setminus\{a_k\}$, \[|D^{\alpha}f_k(x)| \leq (B(||K||+a))^{|\alpha|}\left(||x-a_k||^{-2(|\alpha|+1)}+1\right)|\alpha|!\text{;}\]
\item for all $n \geq 1$, \[\left|\frac{\partial^{2n}f_k}{\partial x_1^{2n}}(a_k)\right| \geq \recip{2^n}(2n)!M^k_n\text{;}\]
\end{enumerate}
where $B$ does not depend on $k$ or the choice of compact set $K$.

Define \begin{equation}f(x) := \sum_{k=1}\recip{2^k} f_k(x).\end{equation}
We will show that $f$ satisfies all the required properties.

The proof that $f \in \CC^\infty(\R^p)$ and that we can differentiate term-by-term is the same, \emph{mutatis mutandis}, as the \hyperref[pf:formaleverywherenot:smooth]{proof of Theorem}~\ref{thm:formaleverywherenot}\satref{sat:formaleverywherenot:CM} (the difference being that here the estimates must be made on compact sets and that there are more coefficients and several extra terms to keep track of).

The proof of \satref{sat:masterthm:CM} is also the same, \emph{mutatis mutandis}, as \hyperref[pf:formaleverywherenot:CM]{the proof of Theorem}~\ref{thm:formaleverywherenot}\satref{sat:formaleverywherenot:CM} (with the same differences as above).

The proofs of \satref{sat:masterthm:polyCM} and \satref{sat:masterthm:quadCM} are similar to each other, and are both similar to \hyperref[pf:formaleverywherenot:smooth]{proof of Theorem}~\ref{thm:formaleverywherenot}\satref{sat:formaleverywherenot:formal}. Fix $m \geq 1$ an integer, and $a > 0$ a real number. Let $\mathcal S := \mathcal S_{a,m}^p$. If $x \neq 0$, then by \satref{sat:masterthm:CM}, we have the desired bounds locally around $x$ in $\mathcal S$. If $x=0$, then by Lemma~\ref{lemma:Sdistance}, for all but finitely many $k$ (say, for $k \geq j$), 
\[\dist(a_k,\mathcal S) \geq \recip{b_k}.\]
Then, there is a bounded neigbhourhood $U$ of $0$ in $\mathcal S$ (i.e. the intersection of a neighbourhood of $0$ in $\R^p$ with $\mathcal S$) such that for all $y \in U$ and $k < j$, $||y-a_k|| > \delta$. Set $C := \max(\delta^{-1},1)$. Let $K$ be any compact set containing $U$. Then, for any $\alpha$ with $|\alpha| \geq 1$, and any $y \in U$,
\begin{align*}
&|D^{\alpha}f(y)| \leq (B(||K||+a))^{\Alp}\Alp!\left(\sum_{k=1}^{j-1}\recip{2^k}(||y-a_k||^{-2(\Alp+1)}+1) \right. \\
&\qquad\qquad\qquad\qquad\qquad\left. + \sum_{k=j}^{\Alp}\recip{2^k}(||y-a_k||^{-2(\Alp+1)}+1) + \sum_{k=\Alp+1}^\infty \recip{2^k}M^k_{\Alp}\right)\\
&\leq (B(||K||+a))^{\Alp}\Alp!\left((j-1)\delta^{-2(\Alp+1)} + \sum_{k=j}^{\Alp}\left(b_k^{1/4}\right)^{2(\Alp+1)} + \sum_{k=1}^{\infty} \recip{2^k} \right)\\
&\leq (B(||K||+a))^{\Alp}\Alp!\left(e^j\delta^{-4\Alp} + \sum_{k=j}^{\Alp}\left(b_k^{1/4}\right)^{4\Alp} + 1\right)\\
&\leq (B(||K||+a))^{\Alp}\Alp!\left(e^j\delta^{-4\Alp}  + \sum_{k=j}^{\Alp}M_k^{\Alp/k}+e^j\right)\\
&\leq (B(||K||+a))^{\Alp}\Alp!\left(e^j\delta^{-4\Alp} + e^{\Alp}M_{\Alp}+e^j\right)\\
&\leq (3e^j)(eBC^4(||K||+a))^{\Alp}\Alp!M_{\Alp}.
\end{align*}

The proof of \satref{sat:masterthm:quadCM} is nearly identical to the proof of \satref{sat:masterthm:polyCM}. Let $K$ be any compact subset of $\R^p \setminus \mathcal Q^p$. Then, for all $x = (x_1,x_2,\ldots,x_p) \in K$, and $k\geq 1$ (considering the cases $x_1 \leq 0$ and $x_2 \leq 0$ separately), $||x-a_k|| \geq \recip{b_k^{1/4}}$. So, for $\Alp \geq 1$, and all $x \in K$,
\begin{align*}
|D^{\alpha}f(x)| &\leq (B(||K||+a))^{\Alp}\Alp!\left(\sum_{k=1}^{\Alp}\recip{2^k}(||y-a_k||^{-2(\Alp+1)}+1) + \hspace{-1ex}\sum_{k=\Alp+1}^\infty \recip{2^k}M^k_{\Alp}\right)\\
&\leq (B(||K||+a))^{\Alp}\Alp!\left(\sum_{k=1}^{\Alp}\left(b_k^{1/4}\right)^{4\Alp} + \sum_{k=1}^\infty \recip{2^k}\right)\\
&\leq (B(||K||+a))^{\Alp}\Alp!\left(\Alp M_{\Alp} + 1\right) \leq 2(eB(||K||+a))^{\Alp}\Alp!M_{\Alp}.
\end{align*}

The proof of \satref{sat:masterthm:notCM} is similar to \hyperref[pf:formaleverywherenot:notCM]{that of Theorem}~\ref{thm:formaleverywherenot}\satref{sat:formaleverywherenot:notCM}. Note that for $n \geq 1$,
\begin{align}
\left|\frac{\partial^{2n}f}{\partial x_1^{2n}}(a_n)\right| &\geq \recip{2^n}\recip{2^n}M^n_n(2n)! - \sum_{k \neq n} \recip{2^k}(B||a_n||+a)^{2n}(2n)!(||a_n-a_k||^{-2(2n+1)}+1)\nonumber\\
&=\label{eq:lowerboundestimatetwo}\recip{4^n}c_nM_n(2n)! - \sum_{k \neq n} \recip{2^k}(B||a_n||+a)^{2n}(2n)!(||a_n-a_k||^{-2(2n+1)}+1).
\end{align}
Since
\begin{align*}
&\ \sum_{k \neq n} \recip{2^k}(B||a_n||+a)^{2n}(2n)!(||a_n-a_k||^{-2(2n+1)}+1)\\
&\leq (Ba+a)^{2n}(2n)!\left(\left(\inf_{n\neq k} |a_n-a_k|\right)^{-2(2n+1)}+1\right) < \infty,\end{align*}
we can choose $c_n \geq M_n$ large so that \eqref{eq:lowerboundestimatetwo} is bigger than $(2n)^{2n}(2n!M_{2n})$, and hence
\begin{align*}
\left|\frac{\partial^{2n}f}{\partial x_1^{2n}}(a_n)\right| \geq (2n)^{2n}(2n)!M_{2n}.
\end{align*}
So, if $f \in \CC^M(\R^p)$, then on some neighbourhood of $0$, there would be $C,D > 0$ such that, for all $n$ and $x \in U$,
\[\left|\frac{\partial^{2n}f}{\partial x_1^{2n}}(x)\right| \leq CD^{2n}(2n)!M_{2n}.\]
But since $a_n \to 0$, for all but finitely many $n$,
\[(2n)^{2n}(2n)!M_{2n} \leq \left|\frac{\partial^{2n}f}{\partial x_1^{2n}}(a_n)\right| \leq CD^{2n}(2n)!M_{2n},\]
which is an obvious contradiction.
\end{proof}

\begin{proof}[Proof of Theorem~\ref{thm:bigtheorem}]
The case $\F=\R$ follows immediately from the case $\F = \C$ by considering real and imaginary parts.
We show that in the complex case, the function $f$ provided by Theorem~\ref{thm:masterthm} satisfies the necessary properties. We know that $f \in \CC^\infty(\R^p)$ and $f \not \in \CC^M(\R^p)$. Let $\gamma \in \CC^M(U,\R^p)$ ($U \subseteq \R$ open) be an arbitrary quasianalytic curve. It is required to show that $f\circ \gamma \in \CC^M(U)$. This is equivalent to showing that for each $t_0 \in U$, there is some $\epsilon > 0$ such that $f\circ \gamma \in \CC^M((t_0-\epsilon,t_0+\epsilon))$. If $\gamma(t_0) \neq 0$, then there is some $\epsilon > 0$ such that $\gamma(t) \neq 0$, for $t \in (t_0-\epsilon,t_0+\epsilon)$. Then, since $\gamma((t_0-\epsilon,t_0+\epsilon)) \subseteq \R^p\setminus\{0\}$, $f\circ\gamma \in \CC^M((t_0-\epsilon,t_0+\epsilon))$, by Theorem~\ref{thm:composition}.

So, it remains to consider the case $\gamma(t_0) = 0$. Without loss of generality, suppose $t_0 = 0$. We distinguish several cases:
\begin{enumerate}[label = (\roman*)]
\item $\gamma_1^{(n)}(0) = 0$, for all  $n\geq 0$;
\item $\gamma_2^{(n)}(0) = 0$, for all $n\geq 0$;
\item $\gamma_1^{(n_1)}(0) \neq 0$ and $\gamma_2^{(n_2)}(0) \neq 0$, for $n_1,n_2 \in \N$.
\end{enumerate}
In the first case, by quasianalyticity, there is  $\epsilon > 0$ such that $\gamma_1|_{(-\epsilon,\epsilon)} \equiv 0$, and as such $|\gamma_2(t)| \geq |\gamma_1(t)|$, for all $|t| < \epsilon$, so that $\gamma((-\epsilon,\epsilon)) \subseteq (\R^p \setminus \mathcal Q^p)\un\mathcal S_{1,1}^p$, and thus $f \circ \gamma \in \CC^M((-\epsilon,\epsilon))$, by Theorem~\ref{thm:composition}.

In the second case,  by quasianalyticity, there is $\epsilon > 0$ such that $\gamma_2|_{(-\epsilon,\epsilon)} \equiv 0$, and as such, $\gamma((-\epsilon,\epsilon)) \subseteq \R^p\setminus \mathcal Q^p$, and thus $f \circ \gamma \in \CC^M((-\epsilon,\epsilon))$, by Theorem~\ref{thm:composition}.

In the third case, let $k_i$ ($i=1,2$) be the smallest integer such that $\gamma_i^{(k_i)}(0) \neq 0$ (note that each $k_i \geq 1$). Then we can write $\gamma_i(t) = t^{k_i}\delta_i(t)$, for $\delta_i:U\to\R$ continuous, and $\delta_i(0) \neq 0$ (by L'H\^{o}pital's rule). For $\epsilon \leq 1$ small, we can assume that there are constants $a_1,a_2 > 0$ such that $|\delta_1(t)| \leq a_1$ and $|\delta_2(t)| \geq a_2$, for $|t| < \epsilon$.
Let $m$ be any integer at least as big as $k_2/k_1$. Then,
\[\frac{a_2}{a_1^{m}}|\gamma_1(t)|^m = \frac{a_2}{a_1^{m}}|t^{k_1}\delta_1(t)|^m \leq a_2|t|^{k_2} \leq |t^{k_2}\delta_2(t)| = |\gamma_2(t)|,\]
so that \[\gamma((-\epsilon,\epsilon)) \subseteq (\R^p \setminus \mathcal Q^p)\un\mathcal S_{a_2/a_1^m,m}^p,\] and thus $f \circ \gamma \in \CC^M((-\epsilon,\epsilon))$, by Theorem~\ref{thm:composition}.
\end{proof}

\begin{rk}In Theorem~\ref{thm:bigtheorem}, that the function can be taken to be of class $\CC^\infty$ is somewhat surprising, as in the analytic case, a function which is smooth and analytic even on every straight line is already analytic (see \cite[Thm.\,5.5.31]{analyticfinally}). This means that there is a large loss of control when passing from $\CC^\omega$ to larger quasianalytic Denjoy-Carleman classes: the extra assumption of smoothness no longer suffices to recover global quasianalyticity from quasianalyticity on every curve.\end{rk}

\begin{rk}Of course it does not make sense to strengthen the hypotheses of Theorem~\ref{thm:bigtheorem} to requiring that $f$ is $\CC^M$ on every $\CC^\infty$ curve: if $\gamma(t)$ is any $\CC^\infty$ curve that is flat at a point $t=0$, then $f\circ \gamma$ is also flat, and is therefore constant by quasianalyticity. Looking at the composition of $f$ with all flat curves $\gamma$ then implies that $f$ is itself constant, too.\end{rk}

\textit{Acknowledgements.} The author's research was conducted as an NSERC Undergraduate Summer Research project under the supervision of Edward Bierstone. The author would like to thank Dr. Bierstone for raising the question treated in Theorem~\ref{thm:bigtheorem} and for his numerous suggestions for this article. The author is grateful to both Dr. Bierstone and Andr\'{e} Belotto for helping him develop his ideas. The author would also like to thank Armin Rainer and David Nenning for pointing out helping correct numerous errors in the first draft of this article, and for the referee's valuble comments on the second draft.

\vspace{1.2ex}
\footnotesize \textsc{University of Toronto, Department of Mathematics, 40 St. George Street,}\\
\textsc{Toronto, ON, Canada M5S 2E4}

\textit{E-mail address}: \texttt{\href{mailto:eyjaffe@mail.com}{eyjaffe@gmail.com}, \href{mailto:ethan.jaffe@mail.utoronto.ca}{ethan.jaffe@mail.utoronto.ca}}

\end{document}